\documentclass[11pt,draft]{amsart}
\usepackage{times}
\usepackage{dsfont}

\author{L. L. de Lima}
\address{Departamento de Matem\'atica\\ Universidade Federal do Cear\'a\\ Brazil}
\author{P. Piccione}
\address{Departamento de Matem\'atica\\ Universidade de S\~ao Paulo\\ Brazil}
\author{M. Zedda}
\address{Universit\`a degli Studi di Cagliari\\ Italy}
\curraddr{Departamento de Matem\'atica\\ Universidade de S\~ao Paulo\\ Brazil}
\title[Bifurcation of solutions of the Yamabe problem]{On bifurcation of solutions of the Yamabe problem\\ in product manifolds}
\date{Revised version of October 18th, 2011}
\thanks{The first author is partially sponsored by CNPq and Funcap, Brazil. The second author is partially sponsored
by CNPq and Fapesp, Brazil. The third author is supported by \emph{RAS}
through a grant financed with the ``Sardinia PO FSE 2007-2013'' funds and
provided according to the L.R.\ 7/2007.}
\begin{document}

\theoremstyle{plain}\newtheorem*{teon}{Theorem}
\theoremstyle{definition}\newtheorem*{defin*}{Definition}
\theoremstyle{plain}\newtheorem{teo}{Theorem}[section]
\theoremstyle{plain}\newtheorem{prop}[teo]{Proposition}
\theoremstyle{plain}\newtheorem*{prop_n}{Proposition}
\theoremstyle{plain}\newtheorem{lem}[teo]{Lemma}
\theoremstyle{plain}\newtheorem{cor}[teo]{Corollary}
\theoremstyle{definition}\newtheorem{defin}[teo]{Definition}
\theoremstyle{remark}\newtheorem{rem}[teo]{Remark}
\theoremstyle{plain} \newtheorem{assum}[teo]{Assumption}
\swapnumbers
\theoremstyle{definition}\newtheorem{example}{Example}[section]
\theoremstyle{plain}\newtheorem*{acknowledgement}{Acknowledgements}
\theoremstyle{definition}\newtheorem*{notation}{Notation}

\maketitle
\begin{abstract}
We study local rigidity and multiplicity of constant scalar curvature metrics in arbitrary products of compact manifolds.
Using (equivariant) bifurcation theory we determine the existence of infinitely many metrics
that are accumulation points of pairwise non homothetic solutions of the Yamabe problem.
Using local rigidity and some compactness results for solutions of the Yamabe problem, we also exhibit new examples
of conformal classes (with positive Yamabe constant) for which uniqueness holds.
\end{abstract}

\begin{section}{Introduction}
The classical Yamabe problem asks for the existence of constant scalar curvature metrics in any given conformal class
of Riemannian metrics on a compact manifold $M$. These metrics can be characterized variationally as critical points
of the Hilbert--Einstein functional on conformal classes. The solution of Yamabe's problem, due to combined efforts
of Yamabe \cite{Yam60}, Trudinger \cite{Tru68}, Aubin \cite{Aub76} and Schoen \cite{Schoen84}, provides \emph{minimizers}
of the Hilbert--Einstein functional in each conformal
class. For instance, Einstein metrics are minima of the functional in their conformal class and in fact,
except for round metrics on spheres, they are the unique metrics in their conformal class having constant scalar curvature,
by a theorem of Obata \cite{Oba72}. It is also interesting to observe that, generically, minima of the
Hilbert--Einstein functional in conformal classes are unique, see \cite{And05}.
However, in many cases a rich variety of constant scalar curvature metrics arise as critical points that are not
necessarily minimizers, and it is a very interesting question to classify all critical points.
In this paper, we propose to use bifurcation theory to determine the existence of multiple constant
scalar curvature metrics on products of compact manifolds.
Multiplicity of solutions of the Yamabe problem in product manifolds has been studied in the literature,
and several results have been obtained in the special case of products with round spheres, see for instance
\cite{HebVau, Koba87, Petean08, Schoen87}. A somewhat different multiplicity result can be found in \cite{Pol93};
bifurcation theory is used in \cite{JinLiXu} to obtain a multiplicity result for the Yamabe equation on the sphere
$\mathds S^N$.
In this paper we consider products of arbitrary factors with constant, but
not necessarily positive, scalar curvature, and we prove a multiplicity result in an infinite number of conformal classes.

Let us describe our result more precisely. Given compact Riemannian manifolds $(M_0,\mathbf g^{(0)})$ and
$(M_1,\mathbf g^{(1)})$, both having positive constant scalar curvature, one consider the \emph{trivial} path
$\mathbf g_\lambda$, $\lambda\in\left]0,+\infty\right[$, of constant scalar curvature metrics on the product
$M=M_0\times M_1$ defined by $\mathbf g_\lambda=\mathbf g^{(0)}\oplus\lambda\,\mathbf g^{(1)}$.
The main results of the paper (Corollary~\ref{thm:cordeginstants}, Theorem~\ref{thm:Yamabebif}) state that there is a
countable set $\Lambda\subset\left]0,+\infty\right[$ that accumulates (only) at $0$ and at $+\infty$ such that:
\begin{itemize}
\item the family $(\mathbf g_\lambda)_\lambda$ is \emph{locally rigid} at all points in $\left]0,+\infty\right[\setminus\Lambda$, i.e.,
for all $\lambda\in\left]0,+\infty\right[\setminus\Lambda$, any constant scalar curvature metric $\mathbf g$ on $M$
which is sufficiently $\mathcal C^{2,\alpha}$-close to $\mathbf g_\lambda$ must be homothetic to some element of the trivial family;
\smallskip

\item at all $\lambda_*\in\Lambda$, except for a finite subset, there is a bifurcating branch of constant scalar curvature metrics issuing from the trivial
branch at $\mathbf g_{\lambda_*}$, and that consists of metrics that do not belong to the trivial family.
\end{itemize}
Rigidity and bifurcation results are also given when the scalar curvatures of $\mathbf g^{(0)}$ and of $\mathbf g^{(1)}$
are not both positive, see Theorem~\ref{thm:nonpositivecase}.
Based on these results and other known facts about Yamabe metrics, one obtains some uniqueness and multiplicity results for
constant scalar curvature metrics in fixed conformal classes, see Section~\ref{sub:multiplicity}.
For instance, an interesting consequence of our bifurcation result yields the following: if $(M_1,\mathbf g^{(1)})$
has positive scalar curvature, then there is a subset $F\subset\left]0,1\right]$ that has a countable
number of accumulation points tending to $0$ such that for all $\lambda\in F$, there are at least \emph{three} distinct constant unit volume scalar curvature metrics
in the conformal class of $\mathbf g_\lambda=\mathbf g^{(0)}\oplus\lambda\,\mathbf g^{(1)}$, see Proposition~\ref{thm:3CSC}.
Finally, using some recent compactness results for solutions of the Yamabe problem, see \cite{KhuMarSch, LiZhang, Mar}, we establish
also the uniqueness of constant scalar curvature metrics in conformal classes in product of spheres or, more generally,
in product of compact Einstein manifolds with positive scalar curvature, see Subsection~\ref{sub:productofspheres}.
\smallskip

The result is obtained as an application of a celebrated abstract bifurcation result of Smoller and Wasserman \cite{SmoWas},
which uses an assumption on the jump of the Morse index for a path of solutions of a family of variational problems.
In the present paper, we consider the variational structure of the Yamabe problem given by the Hilbert--Einstein functional,
defined on the set of metrics of volume $1$ in a given conformal class of metrics.
A very interesting observation on the Yamabe variational problem considered here is that the set $\Lambda$ consisting
of instants when the second variation of the Hilbert--Einstein functional degenerates do not correspond necessarily to
jumps of the Morse index. Namely, the eigenvalues of the Jacobi operator for this functional are arranged into sequences
of functions that are monotonic with respect to $\lambda$, but both increasing and decreasing functions appear, see Lemma~\ref{thm:zerossigmaij}.
Thus, one can have a finite number of degeneracy instants $\lambda\in\Lambda$ where a compensation occurs, and no jump of the Morse index is produced
by the passage through $0$ of the eigenvalues. This raises an extremely interesting question on whether one can have
local rigidity also at this sort of \emph{neutral} degeneracy instants. In the last part of the paper we study this
question, and we setup an equivariant framework to determine some sufficient conditions that guarantee bifurcation
at \emph{every} degeneracy instant. We define the notion of \emph{harmonically freeness} for an isometric action of a Lie
group $G$ on a Riemannian manifold $M$, see Definition~\ref{thm:defnontrivialgroupaction},
which roughly speaking means that the corresponding isotropic representations of $G$ on
distinct eigenspaces of the Laplace--Beltrami operator of $M$ should be direct sum of non equivalent irreducible representations.
The class of manifolds that admit a harmonically free isometric action of a Lie group includes, for instance,
all compact symmetric spaces of rank $1$, see Example~\ref{exa:symrank1}.
We obtain the result that, if one of the two factors $M_0$ or $M_1$ admits a harmonically free isometric action of some
Lie group, then bifurcation of the family $(\mathbf g_\lambda)_\lambda$ must occur at every degeneracy instant (Proposition~\ref{thm:bifurcationneutral}).
This  is obtained using the equivariant abstract bifurcation result of Smoller and Wasserman \cite{SmoWas}, by studying
the representations of the Lie group $G$ on the eigenspaces of the Jacobi operator of the Hilbert--Einstein functional.
\smallskip

The paper is organized as follows. Section~\ref{sec:varsetting} contains the essential facts on the variational framework of
the constant scalar curvature problem in Riemannian manifolds; the basic references for details are \cite{Besse, LeBrun99, Schoen87}.
Section~\ref{sec:locrigbifYamabeabstr} contains statements and proofs of a local rigidity theorem (implicit function theorem)
and both the simple and the equivariant bifurcation result for the Yamabe variational problem.
In Section~\ref{thm:bifproduct} we study explicitly the case of product manifolds and prove our main results.
Appendix~\ref{sec:appfiber} contains formal statements of an implicit function theorem and of two bifurcation
theorems for variational problems defined on the total space of a fiber bundle, which are best suited for
the theory developed in this paper.
\smallskip

It is a pleasure to thank Jimmy Petean, Fernando Cod\'a Marques and Renato Ghini Bettiol
for giving many useful suggestions on a preliminary version of our manuscript.
\end{section}

\begin{section}{The variational setting for the Yamabe problem}
\label{sec:varsetting}
We will denote throughout by $M$ a compact manifold without boundary, with $m=\mathrm{dim}(M)\ge3$, and by $\mathbf g_\mathrm R$
 an auxiliary Riemannian metric on $M$. The metric $\mathbf g_\mathrm R$ induces inner products and norms in all spaces of tensors on $M$,
the Levi--Civita connection $\nabla_\mathrm R$ of $\mathbf g_\mathrm R$ induces a connection in all vector spaces of tensors fields on $M$.
Let $\mathcal S^k(M)$ be the space of all symmetric $(0,2)$-tensors of class $\mathcal C^k$ on $M$, with
$k\ge2$; this is a Banach space when endowed with the norm:
\[\Vert\tau\Vert_{\mathcal C^k}=\max_{j=0,\ldots,k}\left[\max_{p\in M}\big\Vert{\nabla_{\mathrm R}}^{(j)}\tau(p)\big\Vert_\mathrm R\right].\]
Let $\mathcal M^k(M)$ denote the open cone of $\mathcal S^k(M)$ consisting of all Riemannian metrics on $M$;
for all $\mathbf g\in\mathcal M^k(M)$, the tangent space $T_\mathbf g\mathcal M^k(M)$ is identified with the Banach space
$\mathcal S^k(M)$.
Given $\mathbf g\in\mathcal M^k(M)$,
the \emph{conformal class of $\mathbf g$}, denoted by $\big[\mathbf g\big]_k$ is the subset of $\mathcal M^k(M)$ consisting of metrics
that are conformal to $\mathbf g$. For all $\mathbf g$, $\big[\mathbf g\big]_k$ is an open subset of a Banach subspace of $\mathcal S^k$, and
thus it inherits a natural differential structure. As a matter of facts, in order to comply with certain Fredholmness assumptions in Bifurcation Theory, we
need to introduce conformal classes of metrics having a H\"older type regularity $\mathcal C^{k,\alpha}$. To this aim, the most convenient
definition is to consider a smooth\footnote{%
In fact, in most situations it will suffice to assume regularity $\mathcal C^{k+1}$ for $\mathbf g$.} metric $\mathbf g$ on $M$, and setting:
\[\big[\mathbf g\big]_{k,\alpha}=\big\{\psi\cdot\mathbf g:\psi\in\mathcal C^{k,\alpha}(M),\ \psi>0\big\};\]
thus, $\big[\mathbf g\big]_{k,\alpha}$ can be identified with the open subset of the Banach space $\mathcal C^{k,\alpha}(M)$ consisting
of positive functions.
The differential structure on $\big[\mathbf g\big]_{k,\alpha}$ is the one induced by $\mathcal C^{k,\alpha}(M)$.

For $\mathbf g\in\mathcal M^k(M)$, we will denote by $\nu_\mathbf g$ the volume form (or density, if $M$ is not orientable) of $\mathbf g$, by $\mathrm{Ric}_\mathbf g$
the Ricci curvature of $\mathbf g$, and by $\kappa_\mathbf g$
its scalar curvature function, which is a function of class $\mathcal C^{k-2}$ on $M$.

The \emph{volume function} $\mathcal V$ on $\mathcal M^k(M)$ is defined by:
\[\mathcal V(\mathbf g)=\int_M\nu_\mathbf g.\]
Observe that $\mathcal V(\mathbf g)$ is smooth, and its differential is given by:
\begin{equation}\label{eq:difvolume}
\mathrm d\mathcal V(\mathbf g)\mathbf h=\tfrac12\int_M\mathrm{tr}_\mathbf g(\mathbf h)\,\nu_\mathbf g,
\end{equation}
for all $\mathbf h\in\mathcal S^k(M)$.
Let $\mathcal M_1^k(M)$ denote the subset of $\mathcal M^k(M)$ of those metrics $\mathbf g$ such that $\mathcal V(g)=1$;
let us also consider the scale-invariant \emph{Hilbert--Einstein functional} on $\mathcal M^k(M)$, which is
the function $\mathcal A:\mathcal M^k(M)\to\mathds R$ defined by:
\[\mathcal A(\mathbf g)=\mathcal V(\mathbf g)^{\frac{2-m}{m}}\int_M\kappa_{\mathbf g}\,\nu_{\mathbf g}.\]
We summarize here some well known facts about the critical points of $\mathcal A$:
\begin{prop}\label{thm:summainfactscritpt}\hfill
\begin{itemize}
\item[(a)] $\mathcal M_1^k(M)$ is a smooth embedded codimension $1$ submanifold of $\mathcal M^k(M)$.
\smallskip

\item[(b)] $\mathcal M_1^{k,\alpha}(M,\mathbf g)=\mathcal M_1^k(M)\cap\big[\mathbf g\big]_{k,\alpha}$
is a smooth embedded codimension $1$ submanifold of $\big[\mathbf g\big]_{k,\alpha}$.\hfill\break For $\mathbf g_0\in\mathcal M_1^{k,\alpha}(M,\mathbf g)$,
the tangent space $T_{\mathbf g_0}\mathcal M_1^{k,\alpha}(M,\mathbf g)$ is identified with the closed subspace
$\mathcal C^{k,\alpha}_*(M,\mathbf g_0)$ of $\mathcal C^{k,\alpha}(M)$ given by all functions $f$ such that
$\int_Mf\,\nu_{\mathbf g_0}=0$.
\smallskip

\item[(c)] $\mathcal A$ is a smooth functional on $\mathcal M^k(M)$ and on $\big[\mathbf g\big]_{k,\alpha}$.
\smallskip

\item[(d)] the critical points of $\mathcal A$ on $\mathcal M_1^k(M)$ are the \emph{Einstein metrics} of volume $1$ on $M$.
\smallskip

\item[(e)] the critical points of $\mathcal A$ on $\mathcal M_1^{k,\alpha}(M,\mathbf g)$ are those
metrics conformal to $\mathbf g$, having total volume $1$, and that have constant scalar curvature.
\smallskip

\item[(f)] if $\mathbf g_0\in\mathcal M_1^{k,\alpha}(M,\mathbf g)$ is a critical point of $\mathcal A$ on
$\mathcal M_1^{k,\alpha}(M,\mathbf g)$, then the second variation $\mathrm d^2\mathcal A(\mathbf g_0)$ of $\mathcal A$
at $\mathbf g_0$ is identified with the quadratic form on $\mathcal C^{k,\alpha}_*(M,\mathbf g_0)$ defined by:
\begin{equation}\label{eq:secvarformHE}
\mathrm d^2\mathcal A(\mathbf g_0)(f,f)=\frac{m-2}2\int_M\big((m-1)\Delta_{\mathbf g_0}f-\kappa_{g_0}f\big)f\,\nu_{\mathbf g_0}.
\end{equation}
Moreover, $\mathbf g_0$ is a \emph{nondegenerate}\footnote{%
in the sense of Morse theory.} critical point of $\mathcal A$ on
$\mathcal M_1^{k,\alpha}(M,\mathbf g)$ if either $\kappa_{\mathbf g_0}=0$ or if $\frac{\kappa_{g_0}}{m-1}$ is not an eigenvalue
of $\Delta_{\mathbf g_0}$.
\end{itemize}
\end{prop}
\begin{proof}
For $\mathbf g\in\mathcal M_1^k(M)$, setting $\mathbf h=\mathbf g$ in \eqref{eq:difvolume} we get
$\mathrm d\mathcal V(\mathbf g)\mathbf g=\frac12\int_M\mathrm{tr}_\mathbf g(\mathbf g)\,\nu_\mathbf g=
\frac m2\mathcal V(\mathbf g)>0$. Thus, $\mathcal M_1^k(M)$ and $\mathcal M_1^{k,\alpha}(M,\mathbf g)$ are the
inverse image of a regular value of the volume function, which proves (a) and (b).
For $\mathbf g\in\mathcal M_1^k(M)$, the tangent space $T_\mathbf g\mathcal M_1^k(M)$ is the kernel of
$\mathrm d\mathcal V(\mathbf g)$, i.e., the space of those $\mathbf h\in\mathcal S^k(M)$ such that
$\int_M\mathrm{tr}_\mathbf g(\mathbf h)\,\nu_\mathbf g=0$, see \eqref{eq:difvolume}. Setting $\mathbf h=f\cdot\mathbf g$, with
$f\in\mathcal C^{k,\alpha}(M)$, we get $\int_M\mathrm{tr}_\mathbf g(\mathbf h)\,\nu_\mathbf g=m\int_Mf\,\nu_\mathbf g$;
so, the tangent space\footnote{If $\mathbf g_0\in\big[\mathbf g\big]$, then clearly $[\mathbf g_0\big]=\big[\mathbf g\big]$
and $\mathcal M_1^{k,\alpha}(M,\mathbf g_0)=\mathcal M_1^{k,\alpha}(M,\mathbf g)$.
Thus, in this proof it will suffice to consider the case $\mathbf g_0=\mathbf g$} $T_\mathbf g\mathcal M_1^{k,\alpha}(M,\mathbf g)$ is identified with $\mathcal C^{k,\alpha}_*(M,\mathbf g)$.

The smoothness of $\mathcal A$ is clear, since it is the composition of an integral and a second order differential
operator having smooth coefficients. The first variation formula for $\mathcal A$ is given by\footnote{%
The symbol $\langle\cdot,\cdot\rangle_\mathbf g$ in \eqref{eq:firstvarA} denotes the inner product in the space of symmetric
$(0,2)$ tensors induced by $\mathbf g$.}
(see for instance \cite{Schoen87}):
\begin{equation}\label{eq:firstvarA}
\mathrm d\mathcal A(\mathbf g)\mathbf h=-\int_M\left\langle\mathrm{Ric}_\mathbf g-\tfrac12\kappa_\mathbf g\,\mathbf g,\mathbf h\right\rangle_\mathbf g\,\nu_\mathbf g,
\end{equation}
$\mathbf h\in T_\mathbf g\mathcal M_1^k(M)$, from which it follows that $\mathbf g\in\mathcal M_1^k(M)$ is a critical point of $\mathcal A$ if and only if $\mathrm{Ric}_\mathbf g-\tfrac12\kappa_\mathbf g\,\mathbf g=\lambda\cdot\mathbf g$ for some map $\lambda$,
i.e., if and only if exists a function $\mu$ such that $\mathrm{Ric}_\mathbf g=\mu\cdot\mathbf g$.
Taking traces, one sees that $\mu=\frac1m\kappa_\mathbf g$,
i.e., $\mathbf g$ is Einstein. This proves (d). Setting $\mathbf h=f\cdot\mathbf g$ in \eqref{eq:firstvarA}, one
obtains:
\[\mathrm d\mathcal A(\mathbf g)\big(f\cdot\mathbf g)=\tfrac{m-2}2\int_Mf\,\kappa_\mathbf g\,\nu_\mathbf g.\]
This is zero for all $f$  with $\int_Mf\,\nu_\mathbf g=0$ iff and only if $\kappa_\mathbf g$ is constant, proving (e).
Formula \eqref{eq:secvarformHE} can be found, for instance, in \cite{Koiso78, Schoen87}.
It is easy to see that the linear operator $(m-1)\Delta_{\mathbf g}-\kappa_{g}$ is (unbounded) self-adjoint on $L^2(M,\nu_{\mathbf g})$,
that it leaves invariant the set of functions $f$ such that  $\int_Mf\,\nu_{\mathbf g}=0$, and that
its restriction as a linear operator on $\mathcal C^{k,\alpha}_*(M,\mathbf g)$ is Fredholm, and it
has non trivial kernel if and only if $\frac{\kappa_{\mathbf g}}{m-1}$ is a non zero eigenvalue of $\Delta_{\mathbf g}$.
\end{proof}
\begin{rem}\label{thm:remDeltakappa}
An important observation for our theory is that, given $\lambda\in\mathds R^+$,
one has $\Delta_{\mathbf \lambda\mathbf g}=\frac1\lambda\Delta_\mathbf g$ and $\kappa_{\mathbf \lambda\mathbf g}=\frac1\lambda\kappa_\mathbf g$.
This means that the spectrum of the operator $\Delta_\mathbf g-\frac{\kappa_\mathbf g}{m-1}$ is invariant by affine changes of the metric
$\mathbf g$. On the other hand, $\nu_{\mathbf \lambda\mathbf g}=\lambda^\frac m2\nu_\mathbf g$. When needed, we will normalize metrics
to have volume $1$, without changing the spectral theory of the operator $\Delta_\mathbf g-\frac{\kappa_\mathbf g}{m-1}$.
\end{rem}

\end{section}

\begin{section}{Bifurcation and local rigidity for the Yamabe problem}
\label{sec:locrigbifYamabeabstr}
Let $M$ be a fixed compact manifold without boundary, with $\mathrm{dim}(M)=m\ge3$, and assume that $[a,b]\ni\lambda\mapsto\mathbf g_\lambda\in\mathcal S^k(M)$, $k\ge2$, is a continuous path of Riemannian
metrics on $M$ having constant scalar curvature. An element $\lambda_*\in[a,b]$ is a \emph{bifurcation instant} for the family $(\mathbf g_\lambda)_{\lambda\in[a,b]}$
if there exists a sequence $(\lambda_n)_{n\ge1}$ in $[a,b]$ and a sequence $(\mathbf g_n)_{n\ge1}$ in $\mathcal S^k(M)$ of Riemannian metrics on $M$ satisfying:
\begin{itemize}
\item[(a)] for all $n\ge1$, $\mathbf g_n$ belongs to the conformal class of $\mathbf g_{\lambda_n}$, but
$\mathbf g_n\ne\mathbf g_{\lambda_n}$;
\smallskip

\item[(b)] for all $n\ge1$, $\int_M\nu_{\mathbf g_n}=\int_M\nu_{\mathbf g_{\lambda_n}}$;
\smallskip

\item[(c)] for all $n\ge1$, $\mathbf g_n$ has constant scalar curvature;
\smallskip

\item[(d)] $\lim\limits_{n\to\infty}\lambda_n=\lambda_*$ and $\lim\limits_{n\to\infty}\mathbf g_n=\mathbf g_{\lambda_*}$ in $\mathcal S^k(M)$.
\end{itemize}
If $\lambda_*\in[a,b]$ is not a bifurcation instant, then we say that the family $(\mathbf g_\lambda)_\lambda$ is \emph{locally rigid} at $\lambda_*$.
The implicit function theorem provides a sufficient condition for the local rigidity.
\subsection{A sufficient condition for local rigidity}
\begin{prop}\label{thm:localrigidity}
Let $[a,b]\ni\lambda\mapsto\mathbf g_\lambda$ be a smooth path of Riemannian metrics of class $\mathcal C^{k}$, $k\ge3$, having
constant scalar curvature $\kappa_\lambda$ for all $\lambda$, and let $\Delta_\lambda$ denote the Laplace--Beltrami
operator of $\mathbf g_\lambda$. If $\kappa_{\lambda_*}=0$ or if $\frac{\kappa_{\lambda_*}}{m-1}$ is not an eigenvalue
of  $\Delta_{\lambda_*}$ (i.e., if $\mathbf g_{\lambda_*}$ is a nondegenerate critical point of $\mathcal A$ in its conformal class),
then the family $(\mathbf g_\lambda)_\lambda$ is locally rigid at $\lambda_*$.
\end{prop}
\begin{proof}
Up to a suitable normalization, we can assume $\int_M\nu_{\mathbf g_\lambda}=1$ for all $\lambda\in[a,b]$, see Remark~\ref{thm:remDeltakappa}.
Denote by $\mathcal C^{2,\alpha}_+(M)$ the open set of positive functions in $\mathcal C^{2,\alpha}(M)$, and
by $\mathcal D$ the sub-bundle of the trivial fiber bundle $\mathcal C^{2,\alpha}_+(M)\times[a,b]$ over the interval $[a,b]$, defined by:
\begin{equation}\label{eq:defD}
\mathcal D=\left\{(\psi,\lambda)\in\mathcal C^{2,\alpha}_+(M)\times[a,b]:\int_M\psi^\frac m2\,\nu_{\mathbf g_\lambda}=1\right\}.
\end{equation}
Also, let $\mathcal E$ the sub-bundle of $\mathcal C^{0,\alpha}(M)\times[a,b]$ defined by:
\begin{equation}\label{eq:defE}\mathcal E=\left\{(\varphi,\lambda)\in\mathcal C^{0,\alpha}(M)\times[a,b]:\int_M\varphi\,\nu_{\mathbf g_\lambda=0}\right\}.\end{equation}
Finally, consider the smooth map $F:\mathcal D\to\mathcal E$ given by:
\begin{equation}\label{eq:defF}
F(\psi,\lambda)=\left(\kappa_{\psi\cdot\mathbf g_\lambda}-\int_M\kappa_{\psi\cdot\mathbf g_\lambda}\,\nu_{\mathbf g_\lambda},\lambda\right)\in\mathcal E;
\end{equation}
clearly, given $\psi\in\mathcal C^{2,\alpha}_+(M)$ and $\lambda\in[a,b]$, the metric $\psi\cdot\mathbf g_\lambda$ has volume equal to $1$
and constant scalar curvature if and only if $(\psi,\lambda)\in\mathcal D$ and $F(\psi,\lambda)=(0,\lambda)$. This means that, in order to establish the desired result, we need
to look at the structure of the inverse image $F^{-1}(\mathbf 0_\mathcal E)$ of the null section $\mathbf 0_\mathcal E$ of the bundle $\mathcal E$.
Note that $F$ is a fiber bundle morphism, i.e., denoting by $\pi_\mathcal D:\mathcal D\to[a,b]$ and $\pi_\mathcal E:\mathcal E\to[a,b]$ the
natural projections, one has $\pi_\mathcal E\circ F=\pi_\mathcal D$.  The thesis will follows from the Implicit Function Theorem once
we show that the \emph{vertical derivative}\footnote{%
See Appendix~\ref{sec:appfiber}, Proposition~\ref{thm:fiberIFT}.}
$\mathrm d_{\mathrm{ver}}F(\mathbf1,\lambda_*)$ of $F$ at the point $(\mathbf 1,\lambda_*)$
(here $\mathbf1$ is the constant function equal to $1$ on $M$) is a (linear) isomorphism from the Banach space:
\[D_*=\big\{\Psi\in\mathcal C^{2,\alpha}(M):{\textstyle\int}_M\Psi\,\nu_{\mathbf g_{\lambda_*}}=0\big\}\]
to the Banach space
\[E_* =\big\{\Phi\in\mathcal C^{0,\alpha}(M):{\textstyle\int}_M\Phi\,\nu_{\mathbf g_{\lambda_*}}=0\big\}.\]
Observe that $D_*$ is the tangent space at $\psi=\mathbf 1$ of the fiber:
\[\mathcal D_{\lambda_*}=\big\{\psi\in\mathcal C^{2,\alpha}_+(M):{\textstyle\int}_M\psi^\frac m2\,\nu_{\mathbf g_{\lambda_*}}=1\big\}.\]
The vertical derivative $\mathrm d_\mathrm vF(\mathbf1,\lambda_*)$ is easily computed as:
\begin{align}\label{eq:fiberderF}
\nonumber
\tfrac2{m-2}\,\mathrm d_{\mathrm{ver}}F(\mathbf 1,\lambda_*)\Psi&=\!(m-1)\Delta_{\lambda_*}\Psi\!-\!\kappa_{\lambda_*}\Psi\!-\!\int_{\!M}\!\!\Big[(m-1)\Delta_{\lambda_*}\Psi-\kappa_{\lambda_*}\Psi\Big]\,\nu_{\mathbf g_{\lambda_*}}
\\
&=(m-1)\Delta_{\lambda_*}\Psi-\kappa_{\lambda_*}\Psi.
\end{align}
For the second equality above, note that $\Delta_{\lambda_*}$ (as well as the operator given by multiplication by a constant)
carries $D_*$ to $E_*$. Under the assumption that $\kappa_{\lambda_*}=0$ or that $\frac{\kappa_{\lambda_*}}{m-1}$ is not an eigenvalue of
$\Delta_{\lambda_*}$, $\mathrm d_\mathrm fF(\mathbf 1,\lambda_*)$ is injective on $D_*$. Moreover, the linear operator $\Delta_{\lambda_*}-\kappa_{\lambda_*}$
from $\mathcal C^{2,\alpha}(M)$ to $\mathcal C^{0,\alpha}(M)$ is Fredholm of index $0$. Since the codimensions of $D_*$ in $\mathcal C^{2,\alpha}(M)$
and of $E_*$ in $\mathcal C^{0,\alpha}(M)$ are equal (both equal to $1$), it follows that $\mathrm d_\mathrm fF(\mathbf 1,\lambda_*)$ is
an isomorphism from $D_*$ to $E_*$. This concludes the proof.
\end{proof}
\begin{cor}\label{thm:corrigEinst}
If $\mathbf g_{\lambda_*}$ is an Einstein metric which is not the round metric on a sphere, then the family $(\mathbf g_\lambda)_\lambda$ is locally rigid at $\lambda_*$.
\end{cor}
\begin{proof}
By \cite[Theorem~2.4]{Koiso78}, the positive eigenvalues of $\Delta_{\lambda_*}$ are strictly larger than $\kappa_{\lambda_*}$ (i.e., $\mathbf g_{\lambda_*}$
is a strict local minimum of the Hilbert--Einstein functional in its conformal class). The conclusion follows from Proposition~\ref{thm:localrigidity}.
\end{proof}
By a result of B\"ohm, Wang and Ziller, see \cite[Theorem~C, p.\ 687]{BohWanZil}, any metric with unit volume and constant scalar curvature
which is $\mathcal C^{2,\alpha}$-close to an Einstein metric and which is not conformally equivalent to a round metric on the sphere must be a \emph{Yamabe metric},
i.e., it realizes the minimum of the scalar curvature in its conformal class.
Thus, in the situation of Corollary~\ref{thm:corrigEinst}, $\mathbf g_\lambda$ is Yamabe for $\lambda$ near $\lambda_*$. More generally,
$\mathbf g_\lambda$ is a strict local minimum of the Hilbert--Einstein functional in its conformal class for $\lambda$ in every interval $I\subset[a,b]$ containing
$\lambda_*$ such that either $\kappa_\lambda=0$ or $\frac{\kappa_\lambda}{m-1}$ is not an eigenvalue of $\Delta_\lambda$ for all $\lambda\in I$.
For instance, consider the manifold $\mathds S^n$, $n\ge2$, endowed with the standard round metric $\mathbf g$ (say, with normalized volume equal to $1$);
then, the (normalized) product metric $\mathbf g_\lambda=\mathbf g\oplus\lambda\,\mathbf g$ on $\mathds S^n\times\mathds S^n$ is a strict local minimum of the Hilbert--Einstein functional
in its conformal class when $\lambda\in\left]\frac{n-1}n,\frac n{n-1}\right[$, see Subsection~\ref{sub:productofspheres}.

\subsection{Bifurcation of solutions for the Yamabe problem}
An instant $\lambda\in\left]0,+\infty\right[$ for which $\kappa_\lambda\ne0$ and $\frac{\kappa_\lambda}{m-1}$ be an eigenvalue of
$\Delta_\lambda$ will be called a \emph{degeneracy instant} for the family $(\mathbf g_\lambda)_\lambda$. We will now establish some
bifurcation results at the degeneracy instants of $(\mathbf g_\lambda)_\lambda$.

\begin{teo}\label{thm:mainbif}
Let $M$ be a compact manifold, with $\mathrm{dim}(M)=m\ge3$, and let $[a,b]\ni\lambda\mapsto\mathbf g_\lambda\in\mathcal S^k(M)$, $k\ge3$, is a $\mathcal C^1$--path of Riemannian
metrics on $M$ having constant scalar curvature. For all $\lambda\in[a,b]$, denote by $\kappa_\lambda$ the scalar curvature of $\mathbf g_\lambda$, and
by $n_\lambda$ the number of eigenvalues of the Laplace-Beltrami operator $\Delta_\lambda$ (counted with multiplicity) that are
less than $\frac{\kappa_\lambda}{m-1}$. Assume the following:
\begin{itemize}
\item[(a)] $\frac{\kappa_a}{m-1}$ is either equal to $0$, or it is not an eigenvalue of $\Delta_{\mathbf g_a}$;
\smallskip

\item[(b)] $\frac{\kappa_b}{m-1}$ is either equal to $0$, or it is not an eigenvalue of $\Delta_{\mathbf g_b}$;
\smallskip

\item[(c)] $n_a\ne n_b$.
\end{itemize}
Then, there exists a bifurcation instant $\lambda_*\in\left]a,b\right[$ for the family $(\mathbf g_\lambda)_\lambda$.
\end{teo}
\begin{proof}
The result is obtained applying the non equivariant bifurcation theorem \cite[Theorem~2.1, p.\ 67]{SmoWas} to the following setup.
We will use a natural fiber bundle extension of this theorem, whose precise statement is given in Appendix~\ref{sec:appfiber}, Theorem~\ref{thm:extSmoWasgeneral}.
Assume as in the proof of Proposition~\ref{thm:localrigidity} that $\int_M\nu_{\mathbf g_\lambda}=1$ for all $\lambda$, see Remark~\ref{thm:remDeltakappa}.
Consider the fiber bundles $\mathcal D$ and $\mathcal E$, given respectively in \eqref{eq:defD} and \eqref{eq:defE},
and let $F:\mathcal D\to\mathcal E$ be the map given in \eqref{eq:defF};
the inverse image by $F$ of the null section $\mathbf 0_\mathcal D$ of $\mathcal D$ contains the constant section $\mathbf 1_\mathcal E=\{\mathbf 1\}\times[a,b]$,
and the desired result is precisely a fiberwise bifurcation result for this setup.
Let $H=L^2(M)$ denote the Hilbertable space of $L^2$-functions on $M$ with respect to any of the measures induced by the volume forms
$\nu_{\mathbf g_\lambda}$; for all $\lambda$, let $H_\lambda$ be the closed subspace of $H$ consisting of functions $\varphi$ such
that $\int_M\varphi\,\nu_{\mathbf g_\lambda}=0$, endowed with the complete inner product $\langle \phi_1,\phi_2\rangle_\lambda=\int_M\phi_1\phi_2\,\nu_{\mathbf g_\lambda}$.
Note that $T_{\mathbf 1}\mathcal D_\lambda$ is the Banach subspace of $\mathcal C^{2,\alpha}(M)$ consisting of maps
$\Phi$ such that $\int_M\Phi\,\nu_{\mathbf g_\lambda}=0$.
The inclusion $\mathcal C^{k,\alpha}(M)\subset\mathcal C^{k-2,\alpha}(M)\subset L^2(M)$ induce inclusions $T_{\mathbf 1}\mathcal D_\lambda\subset\mathcal E_\lambda\subset H_\lambda$ for all
$\lambda$.
The derivative $\mathrm dF(\cdot,\lambda)$ at $\mathbf 1$ is identified with the vertical derivative $\mathrm d_{\mathrm{ver}}F(\mathbf 1,\lambda)$ given
in \eqref{eq:fiberderF}, which is a linear operator from $T_{\mathbf 1}\mathcal D_\lambda$ to $\mathcal E_\lambda$ which is symmetric with respect
to $\langle\cdot,\cdot\rangle_\lambda$. This is a Fredholm operator of index $0$. Namely, recall
that second order self-adjoint elliptic operators acting on sections of Euclidean vector bundles over compact manifolds
are Fredholm maps of index zero from the space of $\mathcal C^{k,\alpha}$-sections to the space of $\mathcal C^{k-2,\alpha}$-sections,
$k\ge2$, see for instance \cite[\S 1.4]{Whi} and \cite[Theorem~1.1]{Whi2}. The spaces
$T_{\mathbf 1}\mathcal D_\lambda$ and $\mathcal E_\lambda$ are codimension $1$ closed subspaces of $\mathcal C^{k,\alpha}(M)$ and
of $\mathcal C^{k-2,\alpha}(M)$ respectively, and $\mathrm d_\mathrm fF(\mathbf 1,\lambda)$ carries $T_{\mathbf 1}\mathcal D_\lambda$ into
$\mathcal E_\lambda$. This implies that the restriction of $\mathrm d_\mathrm fF(\mathbf 1,\lambda)$ to $T_{\mathbf 1}\mathcal D_\lambda$,
with counterdomain $\mathcal E_\lambda$, is Fredholm of index $0$.

Since $\Delta_{\lambda}$ is a positive discrete operator, it follows that $\Delta_\lambda-\frac{\kappa_\lambda}{m-1}$ has spectrum which consists of
a sequence of finite multiplicity eigenvalues, and only a finite number of them is negative.
Note that $T_{\mathbf 1}\mathcal D_\lambda$ is a codimension $1$ closed subspace of $\mathcal C^{k,\alpha}(M)$ that is
orthogonal relatively to $\langle\cdot,\cdot\rangle_\lambda$ to the eigenspace of the first eigenvalue of $\Delta_\lambda-\frac{\kappa_\lambda}{m-1}$,
which consists of constant functions. This implies that the restriction of $\Delta_\lambda-\frac{\kappa_\lambda}{m-1}$ to $T_{\mathbf 1}\mathcal D_\lambda$
has the same eigenvalues of $\Delta_\lambda-\frac{\kappa_\lambda}{m-1}$, except for the first one (given exactly by $-\frac{\kappa_\lambda}{m-1}$),
each of them with the same eigenspace. In particular, jumps of the dimension of the negative eigenspace of $\mathrm d_\mathrm fF(\mathbf 1,\lambda)$
occur precisely when jumps of the dimension of the negative eigenspace of $\Delta_\lambda-\frac{\kappa_\lambda}{m-1}$ occur.

In conclusion, assumptions (a) and (b) imply that $\mathrm d_\mathrm fF(\mathbf 1,\lambda)$ is an isomorphism at $\lambda=a$ and at $\lambda=b$, respectively.
Assuption (c) implies that there is a jump in the dimension of the negative eigenspace of $\mathrm d_\mathrm fF(\mathbf 1,\lambda)$, as $\lambda$ runs
from $a$ to $b$. The discreteness of the spectrum implies the existence of an isolated instant $\lambda_*\in\left]a,b\right[$ where
$\mathrm d_\mathrm fF(\mathbf 1,\lambda_*)$ is singular, and where a jump of the dimension of the negative eigenspace of  $\mathrm d_\mathrm fF(\mathbf 1,\lambda)$
occurs. Bifurcation must then occur at $\lambda_*$, see Theorem~\ref{thm:extSmoWasgeneral}.
\end{proof}
One can give a more general bifurcation result using an equivariant setup.
Assume in the above situation that there exists a (finite dimensional) \emph{nice} (in the sense of \cite{SmoWas}\footnote{%
A group $G$ is nice if, given unitary representations of $G$ on the finite dimensional
inner product spaces $V$ and $W$, assuming that the quotient spaces $D(V)/S(V)$ and $D(W)/S(W)$ have the same equivariant homotopy type as $G$-spaces
($D$ is the unit disk and $S$ is the unit sphere), then the two representations are equivalent. For instance, denoting by $G_0$
the connected component of the identity of $G$, $G$ is nice if either $G/G_0 =\{1\}$ or if $G/G_0$ is the product of a finite number of copies of $\mathds Z_2$, or of a finite number of copies
of $\mathds Z_3$.})
Lie group $G$ of diffeomorphisms of $M$ that preserves
all the metrics $\mathbf g_\lambda$. This means that, denoting by $I_\lambda$ the isometry group of $(M,\mathbf g_\lambda)$,
$G$ is contained in the intersection $\bigcap_{\lambda\in[a,b]}I_\lambda$. It is easy to see that for every $\lambda$ and
every eigenvalue $\rho$ of $\Delta_\lambda$, one has a linear (anti-)representation\footnote{Note that the action of $G$ on $\mathcal M^k(M)$
by pull-back is on the right.} of $\pi_{\lambda,\rho}:G\to\mathrm{GL}(V_{\lambda,\rho})$,
where $V_{\lambda,\rho}$ is the $\rho$-eigenspace of $\Delta_{\mathbf g_\lambda}$. Such a representation is defined by:
\[\pi_{\lambda,\rho}(\phi)f=f\circ\phi,\]
for all $\phi\in G$ and all $f\in V_{\lambda,\rho}$.
For all $\lambda$, let us denote by $\pi^-_\lambda$ the direct sum representation:
\[\pi^-_\lambda=\bigoplus\limits_{\rho\le\frac{\kappa_\lambda}{m-1}}\pi_{\lambda,\rho}\] of $G$ on the vector  space
$V^-_\lambda$ given by the direct sum:
\[V^-_\lambda=\bigoplus_{\rho\le\frac{\kappa_\lambda}{m-1}}V_{\lambda,\rho}.\] Recall that two linear representations $\pi_i:G\to\mathrm{GL}(V_i)$, $i=1,2$, of
the group $G$ on the vector space $V_i$ are \emph{equivalent} if there exists an isomorphism $T:V_1\to V_2$ such that
$\pi_2(g)\circ T=T\circ\pi_1(g)$ for all $g\in G$.

We then have the following extension of Theorem~\ref{thm:mainbif}:
\begin{teo}\label{thm:mnainbifequiv}
In the above situation, assume that:
\begin{itemize}
\item $\frac{\kappa_a}{m-1}$ is either equal to $0$, or it is not an eigenvalue of $\Delta_{\mathbf g_a}$;
\smallskip

\item $\frac{\kappa_b}{m-1}$ is either equal to $0$, or it is not an eigenvalue of $\Delta_{\mathbf g_b}$;
\smallskip

\item $\pi^-_a$ and $\pi^-_b$ are not equivalent.
\end{itemize}
Then, there exists a bifurcation instant $\lambda_*\in\left]a,b\right[$ for the family $(\mathbf g_\lambda)_\lambda$.
\end{teo}
\begin{proof}
This uses the equivariant bifurcation result of \cite[Theorem~3.1]{SmoWas}, applied to the setup described in
the proof of Theorem~\ref{thm:extSmoWasgeneral}. See Theorem~\ref{thm:extSmoWasequivariant} for the precise statement
needed for our purposes. Note that the (right) action of $G$ on $\mathcal D$ is given by $(\psi,\lambda)\cdot\phi=(\psi\circ\phi,\lambda)$,
for all $(\psi,\lambda)\in\mathcal D$ and all $\phi\in G$, similarly for the action of $G$ on $\mathcal E$,
and the function $F$ is equivariant with respect to this action. Clearly, constant functions are fixed by this action, and
the remaining assumptions of Theorem~\ref{thm:extSmoWasequivariant} are easily checked, as in the proof of Theorem~\ref{thm:mainbif}.
\end{proof}
\end{section}
\begin{section}{Bifurcation in product manifolds}
\label{thm:bifproduct}
Let $\big(M_0,\mathbf g^{(0)}\big)$, $\big(M_1,\mathbf g^{(1)}\big)$ be compact Riemannian manifolds with constant scalar curvature denoted by
$\kappa^{(0)}$ and $\kappa^{(1)}$ respectively. Let $m_0$ (resp.\ $m_1$) be the dimension of $M_0$ (resp.\ $M_1$), and assume $m_0+m_1\ge3$. For all
$\lambda\in\left]0,+\infty\right[$ denote by $\mathbf g_\lambda=\mathbf g^{(0)}\oplus \lambda\cdot\mathbf g^{(1)}$ the
metric on $M=M_0\times M_1$. Clearly, $\mathbf g_\lambda$ has constant scalar curvature
\begin{equation}\label{eq:kappalambda}
\kappa_\lambda=\kappa^{(0)}+\tfrac1\lambda\,\kappa^{(1)}.
\end{equation}
Observe that, as to degeneracy instants and bifurcation, the role played by the manifolds $\big(M_0,\mathbf g^{(0)}\big)$ and
$\big(M_1,\mathbf g^{(1)}\big)$ is symmetric. Namely, degeneracy instants and bifurcation instants for the family $(\mathbf g_\lambda)_\lambda$
coincide respectively with degeneracy instants and bifurcation instants for the family of metrics $\mathbf h_\lambda=\frac1\lambda\,\mathbf g^{(0)}\oplus\mathbf g^{(1)}$
on $M=M_0\times M_1$.

Set $m=m_0+m_1=\mathrm{dim}(M)$, and let $\mathcal J_\lambda$ be the Jacobi operator of the Hilbert--Einstein functional along $\mathbf g_\lambda$, given by:
$$\mathcal J_\lambda=\Delta_{\lambda}-\frac{\kappa_{\lambda}}{m-1},$$
defined on the space:
\[\Big\{\Psi\in\mathcal C^{2,\alpha}(M):{\textstyle\int}_M\Psi\,\nu_{\mathbf g_\lambda}=0\Big\}\]
and taking values in the space:
\[\Big\{\Phi\in\mathcal C^{0,\alpha}(M):{\textstyle\int}_M\Phi\,\nu_{\mathbf g_\lambda}=0\Big\};\]
let $\Sigma(\mathcal J_\lambda)$ be its spectrum. This spectrum coincides with the spectrum of $\Delta_{\lambda}-\frac{\kappa_{\lambda}}{m-1}$ as
an operator from $\mathcal C^{2,\alpha}(M)$ to $\mathcal C^{0,\alpha}(M)$, with the point $-\frac{\kappa_{\lambda}}{m-1}$ removed.

Denote by $0=\rho^{(i)}_1<\rho^{(i)}_2<\rho^{(i)}_3<\ldots$ the sequence of eigenvalues of $\Delta_{\mathbf g^{(i)}}$, $i=0,1$,
and denote by $\mu^{(i)}_j$ the multiplicity of $\rho^{(i)}_j$;
Then:
\[\Sigma(\mathcal J_\lambda)=\Big\{\sigma_{i,j}(\lambda): i,j\geq 0, i+j>0\Big\},\]
where:
\begin{equation}\label{eigenij}
\sigma_{i,j}(\lambda)=\rho^{(0)}_i+\tfrac1\lambda\,\rho^{(1)}_j-\frac1{m-1}\left(\kappa^{(0)}+\tfrac1\lambda\,\kappa^{(1)}\right).
\end{equation}
The multiplicity of $\sigma_{i,j}(\lambda)$ in $\Sigma(\mathcal J_\lambda)$  is equal to the product $\mu^{(0)}_i\mu^{(1)}_j$,
note however that the $\sigma_{i,j}$'s need not be all distinct.
Our interest is to determine the distribution of zeros of the functions $\lambda\mapsto\sigma_{i,j}(\lambda)$ as $i$ and $j$ vary;
such zeros correspond to degeneracy instants of the Jacobi operator $\mathcal J_\lambda$.
Towards this goal, we make a preliminary observation.
\begin{rem}\label{thm:rematmost}
Each function $\sigma_{i,j}$ which is not identically zero
has at most one zero in $\left]0,+\infty\right[$. Moreover,
for any fixed $i$ and $\overline\lambda\in\left]0,+\infty\right[$, there is at most one $j$ for which $\sigma_{i,j}(\overline\lambda)=0$.
This depends on the fact that the sequence $j\mapsto\rho^{(1)}_j$ is strictly increasing. Similarly, for each $j$ and $\overline\lambda\in\left]0,+\infty\right[$, there is at most one value of $i$ for which $\sigma_{i,j}(\overline\lambda)=0$.
\end{rem}

Let $i_*$ and $j_*$ be the smallest nonnegative integers with the property that:
\begin{equation}\label{eq:defi*j*}
\rho^{(0)}_{i_*}\ge\frac{\kappa^{(0)}}{m-1},\quad\rho^{(1)}_{j_*}\ge\frac{\kappa^{(1)}}{m-1}.
\end{equation}
Let us say that the pair of metrics $\big(\mathbf g^{(0)},\mathbf g^{(1)}\big)$ is \emph{degenerate}
if equalities hold in both inequalities of \eqref{eq:defi*j*}. In this situation, the Jacobi operator $\mathcal J_\lambda$ is degenerate for
all $\lambda>0$, namely, $\sigma_{i_*,j_*}(\lambda)=0$ for all $\lambda$.
\begin{rem}\label{thm:remdegeneracyEinstein}
Clearly, if either $\kappa^{(0)}<0$ or $\kappa^{(1)}<0$, then $\big(\mathbf g^{(0)},\mathbf g^{(1)}\big)$ is not degenerate.
We observe also that if either one of the two metrics $\mathbf g^{(0)}$ or $\mathbf g^{(1)}$ is Einstein with positive scalar curvature, then
the pair $\big(\mathbf g^{(0)},\mathbf g^{(1)}\big)$ is never degenerate. Namely, if say $\mathbf g^{(0)}$ is Einstein and $\kappa^{(0)}>0$,
then $\kappa^{(0)}=m_0\mathrm{Ric}_{\mathbf g^{(0)}}$; using Lichnerowicz--Obata theorem (see for instance \cite[Ch.~3, \S D]{BerGauMaz},
or \cite{Oba62}) one gets:
\[\rho^{(0)}_1\ge\frac{m_0}{m_0-1}\mathrm{Ric}_{\mathbf g^{(0)}}=\frac{\kappa^{(0)}}{m_0-1}>\frac{\kappa^{(0)}}{m-1}.\]
This says that $i_*=1$, and that equality does not hold in the first inequality of \eqref{eq:defi*j*}.
We note however that when the metrics $\mathbf g^{(0)}$ and  $\mathbf g^{(1)}$ are not Einstein, then the integers $i_*$ and $j_*$
defined above can be arbitrarily large. For instance, given any manifold $(\overline M,\overline{\mathbf g})$ with positive scalar curvature
$\overline\kappa$, then the product Riemannian manifold $M_0=\overline M\times\mathds S^1(r)$, where $\mathds S^1(r)$ is the circle of radius $r>0$,
has constant scalar curvature larger than $\overline\kappa$, and every eigenvalue of its Laplace--Beltrami operator goes to $0$ as $r\to+\infty$.
This implies that $i_*$ becomes arbitrarily large as $r\to+\infty$.
\end{rem}
Except for case of degenerate pairs, the operator $\mathcal J_\lambda$ is singular only at a discrete countable set of instants $\lambda$ in $\left]0,+\infty\right[$.
We consider separately the (most interesting) case that both scalar curvatures $\kappa^{(0)}$ and $\kappa^{(1)}$ are positive.
\subsection{The case of positive scalar curvatures}
\begin{lem}\label{thm:zerossigmaij}
Assume $\big(\mathbf g^{(0)},\mathbf g^{(1)}\big)$ non degenerate, and that $\kappa^{(0)},\kappa^{(1)}>0$. The functions $\sigma_{i,j}(\lambda)$ satisfy the following properties.
\begin{itemize}
\item[(a)]  For all $i,j\ge0$, the map $\lambda\mapsto\sigma_{i,j}(\lambda)$ is strictly monotone in $\left]0,+\infty\right[$, except possibly the
maps $\sigma_{i,j_*}$, that are constant equal to $\rho^{(0)}_i-\frac{\kappa^{(0)}}{m-1}$ when $\rho^{(1)}_{j_*}=\frac{\kappa^{(1)}}{m-1}$.
\smallskip\

\item[(b)]  For $i\ne i_*$ and $j\ne j_*$, the map $\sigma_{i,j}(\lambda)$ admits a zero if and only if:
\begin{itemize}
\item either $j<j_*$ and $i>i_*$, in which case $\sigma_{i,j}$ is strictly increasing,
\item or if $j>j_*$ and $i<i_*$, in which case $\sigma_{i,j}$ is strictly decreasing.
\end{itemize}
\smallskip

\item[(c)] If $\rho^{(0)}_{i_*}=\frac{\kappa^{(0)}}{m-1}$, then $\sigma_{i_*,j}$ does not have zeros for any $j$.
If $\rho^{(0)}_{i_*}>\frac{\kappa^{(0)}}{m-1}$, then $\sigma_{i_*,j}$ has a zero if and only if $j<j_*$.
\smallskip

\item[(d)] If $\rho^{(1)}_{j_*}=\frac{\kappa^{(1)}}{m-1}$, then $\sigma_{i,j_*}$ does not have zeros for any $i$.
If $\rho^{(1)}_{j_*}>\frac{\kappa^{(1)}}{m-1}$, then $\sigma_{i,j_*}$ has a zero if and only if $i<i_*$.
\end{itemize}
\end{lem}
\begin{proof}
The entire statement follows readily from a straightforward analysis of \eqref{eigenij}, writing $\sigma_{i,j}(\lambda)=A_i+\frac1\lambda{B_j}$,
with $A_i=\rho_i^{(0)}-\frac{\kappa^{(0)}}{m-1}$, and $B_j=\rho_j^{(1)}-\frac{\kappa^{(1)}}{m-1}$.
\end{proof}

\begin{cor}\label{thm:cordeginstants}
If $\big(\mathbf g^{(0)},\mathbf g^{(1)}\big)$ is non degenerate, then the set of instants $\lambda$ in the open half line $\left]0,+\infty\right[$ at which the Jacobi operator
is singular is countable and discrete; it consists of a strictly increasing unbounded sequence and a strictly decreasing sequence tending to
$0$. For all other values of $\lambda$, $\mathcal J_\lambda$ is an isomorphism, and in particular,
the family $(\mathbf g_\lambda)_\lambda$ is locally rigid at these instants.
\end{cor}
\begin{proof}
By Lemma~\ref{thm:zerossigmaij}, each function $\sigma_{i,j}$ has at most one zero, thus there is only a countable numbers of degeneracy instants for $\mathcal J_\lambda$.
For $j>j_*$ and $i< i_*$, the zero $\lambda_{i,j}$ of $\sigma_{i,j}$  satisfies:
\[\lambda_{i,j}=\left\vert\frac{B_j}{A_i}\right\vert\ge B_j\cdot\Big[\frac{\kappa^{(0)}}{m-1}-\rho^{(0)}_{i_*-1}\Big]^{-1}\longrightarrow+\infty,\quad\text{as $j\to+\infty$}.\]
Similarly, for $i>i_*$ and $j<j_*$, the zero $\lambda_{i,j}$ of $\sigma_{i,j}$ satisfies:
\[0<\lambda_{i,j}=\left\vert\frac{B_j}{A_i}\right\vert\le A_i^{-1}\cdot\frac{\kappa^{(1)}}{m-1}\longrightarrow0,\quad\text{as $i\to+\infty$}.\]
The conclusion follows.
\end{proof}

\begin{teo}\label{thm:Yamabebif}
Let $\big(M_0,\mathbf g^{(0)}\big)$ and $\big(M_1,\mathbf g^{(1)}\big)$ be compact Riemannian manifolds with positive constant
scalar curvature; assume that the pair $\big(\mathbf g^{(0)},\mathbf g^{(1)}\big)$ is nondegenerate.
For $\lambda\in\left]0,+\infty\right[$, let $\mathbf g_\lambda$ denote the metric $\mathbf g^{(0)}\oplus\lambda\,\mathbf g^{(1)}$
on the product $M_0\times M_1$. Then, there exists a sequence $\big(\lambda^{(1)}_n\big)_n$ tending to $0$ as $n\to\infty$ and a sequence
$\big(\lambda^{(2)}_n\big)_n$ tending
to $+\infty$ as $n\to\infty$ consisting of bifurcation instants for the family $(\mathbf g_\lambda)_\lambda$.
\end{teo}
\begin{proof}
By Corollary~\ref{thm:cordeginstants}, there are two sequences of instants $\lambda$ at which the Jacobi operator
$\mathcal J_\lambda$ is singular; these instants are our candidates to be bifurcation instants. In principle
one cannot guarantee that at each of these instants there is a jump in the dimension of the negative eigenspace
of $\mathcal J_\lambda$; namely, the eigenvalues $\sigma_{i,j}(\lambda)$ described in Lemma~\ref{thm:zerossigmaij} can be
either increasing or decreasing. Nevertheless, the zeroes of those eigenvalues that are increasing functions
accumulate (only) at zero, while the zeroes of those eigenvalues that are decreasing functions accumulate (only)
at $+\infty$. This implies that at all but a finite number of degeneracy instants there is jump of dimension in the
negative eigenspace of $\mathcal J_\lambda$. The conclusion follows then from Theorem~\ref{thm:mainbif}.
\end{proof}
Note that the case of degenerate pairs cannot be treated with Theorem~\ref{thm:mainbif}, because $\mathcal J_\lambda$ is degenerate
for all $\lambda$, and thus assumptions (a) and (b) are never satisfied in this case.
\smallskip

Theorem~\ref{thm:Yamabebif} leaves an open question on whether there may be some degeneracy instants for
the Jacobi operator $\mathcal J_\lambda$ at which bifurcation \emph{does not} occur. In principle, this
situation might occur at those instants $\lambda$ at which two or more eigenvalue functions $\sigma_{i,j}$
vanish, compensating the positive and the negative contributions to the dimension of the negative eigenspace.
Let us call \emph{neutral} a degeneracy instant of this type.
It is quite intuitive that existence of neutral degeneracy instants should not occur generically,
although a formal proof of this fact might be quite awkward.

There is an interesting case in which one can establish bifurcation also at neutral degeneracy instants,
using the equivariant result of Theorem~\ref{thm:mnainbifequiv}. This case is studied in the sequel.
Let us give the following definition:
\begin{defin}\label{thm:defnontrivialgroupaction}
Two representations $\pi_i$, $i=1,2$ of a group $G$ are said to be \emph{essentially equivalent} if
one of the two is equivalent to the direct sum of the other with a number of copies of the trivial
representation of $G$.
Let $G$ be a group acting by isometries on a Riemannian manifold $(N,\mathbf h)$.
The action will be called \emph{harmonically free} if, given an arbitrary family $V_1,\ldots,V_r,V'_1,\ldots,V'_s$
of pairwise distinct eigenspaces of the Laplacian $\Delta_\mathbf h$, then the corresponding representations
of $G$ on the direct sums $V=\bigoplus_{i=1}^rV_i$ and $V'=\bigoplus_{j=1}^sV'_j$ are not essentially equivalent.
\end{defin}
For instance, the natural action of the orthogonal group $\mathrm O(n)$ on the round sphere $\mathds S^{n+1}$ is
harmonically free. Namely, the representation of $\mathrm O(n)$ on each eigenspace of the Laplacian of
$\mathds S^{n+1}$ is irreducible. Moreover, the dimension of the eigenspaces of the Laplacian of $\mathds S^{n+1}$ form
a strictly increasing sequence, from which it follows that the representations of $\mathrm O(n)$ on the
eigenspaces of the Laplacian of $\mathds S^{n+1}$ are pairwise non equivalent. This in particular implies that
direct sum of any two distinct families of eigenspaces of the Laplacian are never essentially equivalent.

\begin{example}\label{exa:symrank1}
More generally, the action of the isometry group of a compact manifold is harmonically free when the
eigenspaces of the Laplacian are irreducible and pairwise non equivalent. An important class of examples
of this situation (see \cite[Ch.~III, \S~C]{BerGauMaz}) is given by the compact symmetric spaces of rank one,
which consists of the following homogeneous spaces $G/H$ with a $G$-invariant metric:
\begin{itemize}
\item the real projective spaces $\mathds RP^k$, with $G=\mathrm O(k+1)$ and $H=\mathrm O(k)\times\{-1,1\}$;
\smallskip

\item the complex projective spaces $\mathds CP^k$, with $G=\mathrm U(k+1)$ and $H=\mathrm U(k)\times\mathrm U(1)$;
\smallskip

\item the quaternionic projective spaces $\mathds HP^k$, with $G=\mathrm{Sp}(k+1)$ and $H=\mathrm{Sp}(k)\times\mathrm{Sp}(1)$;
\smallskip

\item the Cayley plane $\mathds P^2(\mathrm{Ca})$, with $G=F_4$ and $H=\mathrm{Spin}(9)$.
\end{itemize}
In these examples, the eigenspaces of the Laplacian are irreducible by the natural action of $G$, see \cite[Proposition~C.I.8]{BerGauMaz},
and the dimension of these eigenspaces form a strictly increasing sequence. In particular, they are pairwise non equivalent.
Observe also that all these examples have constant scalar curvature, by homogeneity.
In fact, all these examples are \emph{two point homogeneous}, which implies that they are Einstein.
\end{example}

\begin{prop}\label{thm:bifurcationneutral}
Under the hypothesis of Theorem~\ref{thm:Yamabebif}, assume in addition that there exists a nice Lie group
$G$ with an isometric and harmonically free action on either $\big(M_0,\mathbf g^{(0)}\big)$ or on $\big(M_1,\mathbf g^{(1)}\big)$.
Then, every degeneracy instant for the Jacobi operator $\mathcal J_\lambda$ is a bifurcation instant for
the family $(\mathbf g_\lambda)_\lambda$.
\end{prop}
\begin{proof}
We can assume that $G$ acts on $\big(M_0,\mathbf g^{(0)}\big)$. For all $\lambda\in\left]0,+\infty\right[$, one obtains
a non trivial isometric action of $G$ on $(M,\mathbf g_\lambda)$ by setting $g\cdot(x_0,x_1)=(g\cdot x_0,x_1)$, $g\in G$,
$x_0\in M_0$ and $x_1\in M_1$. Let $\overline\lambda$ be a neutral degeneracy instants for the family $\mathbf g_\lambda$, and
let $\sigma_{i,j}$ be one of the eigenvalue functions that vanish at $\overline\lambda$. For all $\lambda$, the eigenspace
of $\sigma_{i,j}(\lambda)$ is the direct sum of the $i$-th eigenspace $V_i$ of $\Delta_{\mathbf g^{(0)}}$ and the $j$-th eigenspaces
$W_j$ of $\Delta_{\mathbf g^{(1)}}$. There is a representation of $G$ on this direct sum, given by the direct sum of the natural representation
of $G$ on the eigenspace $V_i$ of $\Delta_{\mathbf g^{(0)}}$ and the trivial representation of $G$ on $W_j$.
As $\lambda$ increases and crosses $\overline\lambda$, the space $V_i\oplus W_j$ is added or removed from the negative eigenspace of $\mathcal J_\lambda$,
according to whether $\sigma_{i,j}$ is decreasing or increasing.

Denote by $\mathcal H_0$ the direct sum of eigenspaces of those eigenvalues $\sigma_{i,j}$ that are negative on the interval
$[\overline\lambda-\varepsilon,\overline\lambda+\varepsilon]$.
Then, for $\varepsilon>0$ small enough,
the negative eigenspace of $\mathcal J_{\overline\lambda-\varepsilon}$ is a direct sum of the form:
\[\mathcal H_0\oplus\bigoplus_{k=1}^rV_{i_k}\oplus W_{j_k},\] and the negative eigenspace of $\mathcal J_{\overline\lambda+\varepsilon}$ is the direct sum
\[\mathcal H_0\oplus\bigoplus_{l=r+1}^{r+s}V_{i_l}\oplus W_{j_l},\]
where the family $V_{i_1},\ldots,V_{i_r},V_{i_{r+1}},\ldots,V_{i_s}$ consists of pairwise distinct eigen\-spaces
of $\Delta_{\mathbf g^{(0)}}$.
This follows from the fact that if $(i,j)\ne(i',j')$ and $\sigma_{i,j}(\overline\lambda)=\sigma_{i',j'}(\overline\lambda)=0$, then necessarily $i\ne i'$ and $j\ne j'$, see Remark~\ref{thm:rematmost}.
The representation $\pi^-_{\overline\lambda-\varepsilon}$ is
the direct sum of the representations of $G$ on $\mathcal H_0$, on $V=\bigoplus_{k=1}^rV_{i_k}$, plus a number of copies of the trivial representation of $G$,
while he representation $\pi^-_{\overline\lambda+\varepsilon}$ is the direct sum of the representations of $G$
on $\mathcal H_0$, on $V'=\bigoplus_{l=r+1}^{r+s}V_{i_l}$ plus a number of copies of the trivial representation of $G$. Hence, $\pi^-_{\overline\lambda-\varepsilon}$
and $\pi^-_{\overline\lambda+\varepsilon}$ are not equivalent, because the action of $G$ on $\big(M_0,\mathbf g^{(0)}\big)$ is
harmonically free. The result follows then from Theorem~\ref{thm:mnainbifequiv}.
\end{proof}

\begin{cor}\label{thm:prodsphere}
Let $(M_1,\mathbf g^{(1)})$ be a compact symmetric space of rank $1$.
Given any compact Riemannian manifold $(M_0,\mathbf g^{(0)})$ with positive constant scalar curvature, then
the family $\mathbf g_\lambda=\mathbf g^{(0)}\oplus\lambda\,\mathbf g^{(1)}$ on $M_0\times M_1$
has a countable number of degeneracy instants that accumulate at $0$ and at $+\infty$. There is bifurcation at every degeneracy instant.
\end{cor}
\begin{proof}
Set $m_0=\mathrm{dim}(M_0)\ge2$, write $M_1=G/H$, and consider the isometric action of $G$ by left multiplication.
Since compact symmetric spaces of rank $1$ are Einstein and have positive scalar curvature,
then the pair $\big(\mathbf g^{(0)},\mathbf g^{(1)}\big)$ is nondegenerate, see Remark~\ref{thm:remdegeneracyEinstein}.
Finally, observe that all the groups $G$, except for $G=\mathrm O(k+1)$, that appear in Example~\ref{exa:symrank1}, are connected, hence they are nice.
Also the orthogonal group $\mathrm O(k+1)$ is nice, as $\mathrm O(k+1)/\mathrm{SO}(k+1)\cong\mathds Z_2$.
The result now follows from Corollary~\ref{thm:cordeginstants} and Proposition~\ref{thm:bifurcationneutral}, keeping in mind
that the action of  $G$ on $M$ is harmonically free, see Example~\ref{exa:symrank1}.
\end{proof}
\subsection{Product of spheres}\label{sub:productofspheres}
Consider the case when $M$ is the product of two spheres $\mathds S^{n}\times \mathds S^{n}$ of same dimension $n$,
endowed with the metric $\mathbf g_\lambda= \mathbf g\oplus \lambda\, \mathbf g$, where $\mathbf g$ is the standard round metric on $\mathds S^{n}$.
Since $\mathbf g_\lambda$ and $\mathbf g_{\frac{1}{\lambda}}$ belong to the same conformal class, it suffices
to consider the case $\lambda\in\left]0,1\right]$.

The $j$-th eigenvalue of $\Delta_{\mathbf g}$ is $\rho_j=j(j+n-1)$, which gives
\[\sigma_{i,j}(\lambda)=\frac{1}{\lambda}\left[j(j+n-1)-\frac{n(n-1)}{2n-1}\right]+i(i+n-1)-\frac{n(n-1)}{2n-1};\]
by Corollary~\ref{thm:prodsphere}, every zero of $\sigma_{ij}$ is a bifurcation instant.
One computes easily that $\sigma_{i,j}$ has a zero in the interval $\left]0,1\right]$ only if $j=0$;
the zero of $\sigma_{i,0}$
in $\left]0,1\right]$ is given by:
\[\lambda_i(n)=\frac{n(n-1)}{i(i+n-1)(2n-1)-n(n-1)}, \quad i>0;\]
this forms a strictly decreasing sequence tending to $0$ as $i\to+\infty$, and its maximum is $\lambda_1(n)=\frac{n-1}n$.
By Proposition~\ref{thm:localrigidity}, the family $\mathbf g_\lambda$ is locally rigid in the interval $\left]\frac{n-1}n,\frac n{n-1}\right[$.

Since for $\lambda=1$ the metric $\mathbf g_\lambda$ on $\mathds S^n\times\mathds S^n$ is Einstein, we know that $\mathbf g_1$ is the unique metric in its conformal class
with given volume and constant scalar curvature. It is an interesting open question if the same is true for the metric $\mathbf g_\lambda$,
for $\lambda\in\left]\frac{n-1}n,\frac n{n-1}\right[$. Our local rigidity result gives a partial answer to this question,
in that it excludes the
existence of other constant scalar curvature metrics with given volume \emph{near} $\mathbf g_\lambda$ for
$\lambda\in\left]\frac{n-1}n,\frac n{n-1}\right[$.
This result can be improved as follows:
\begin{prop}\label{thm:uniquenessprodspheres}
Consider the product manifold $M=\mathds S^n\times\mathds S^n$ endowed with the metric $\mathbf g_\lambda=\mathbf g\oplus\lambda\cdot\mathbf g$, where
$\mathbf g$ is the round metric on $\mathds S^n$.
Consider the set:
\begin{align}
\mathcal A=\Big\{\lambda\in\left]\tfrac{n-1}n,\tfrac n{n-1}\right[:\ \ &\text{the conformal class of $\mathbf g_\lambda$ contains only one metric}
\\
&\text{ with constant scalar curvature and volume $v_\lambda$}\Big\};
\end{align}
Then, $\mathcal A$ is an open subset of $\left]\tfrac{n-1}n,\tfrac n{n-1}\right[$ containing $1$.

If $\overline\lambda$ is an accumulation point of $\mathcal A$,
then every constant curvature metric in the conformal class of $\mathbf g_{\overline\lambda}$ which is not
homothetic to $\mathbf g_{\overline\lambda}$ is \emph{degenerate}\footnote{%
i.e., a degenerate critical point of the Hilbert--Einstein functional $\mathcal A$ in $\mathcal M_1^{2,\alpha}(M,\mathbf g_{\overline\lambda})$,
see item~(f) in Proposition~\ref{thm:summainfactscritpt}.}.
\end{prop}
\begin{proof}
Clearly $1\in\mathcal A$, as we observed above.
By taking homotheties, we can assume that the volume of each $\mathbf g_\lambda$ is equal to $1$.
Assume $\lambda_*\in\mathcal A$ and, by absurd, that there exists a sequence
$\lambda_k\in\left]\tfrac{n-1}n,\tfrac n{n-1}\right[\setminus\mathcal A$
with $\lim\limits_{k\to\infty}\lambda_k=\lambda_*$. Let $\mathbf g_k$ be a constant scalar curvature metric in the conformal class of $\mathbf g_{\lambda_k}$
and of volume $1$ which is different from $\mathbf g_{\lambda_k}$. By the local rigidity around $\lambda_*$, for $k$ large $\mathbf g_{\lambda_k}$
cannot enter in some neighborhood of $\mathbf g_{\lambda_*}$.
The set of unit volume constant scalar curvature metrics on $\mathds S^n\times\mathds S^n$
that belong to the conformal class of some $\mathbf g_\lambda$, with $\lambda\in\left[\frac{n-1}n,\frac n{n-1}\right]$ is compact
in the $\mathcal C^2$-topology; this follows easily from \cite{KhuMarSch, LiZhang, Mar},
see Proposition~\ref{thm:compactYamabe} below. Hence, the sequence $\mathbf g_k$ must have a subsequence
converging in the $\mathcal C^2$-topology to a metric $\mathbf g_\infty$ which belongs to the conformal class of $\mathbf g_{\lambda_*}$. By continuity,
$\mathrm{vol}(M,\mathbf g_\infty)=1$
and $\mathbf g_\infty$ has constant scalar curvature. This gives a contradiction, because it must be $\mathbf g_\infty\ne\mathbf g_{\lambda_*}$,
but $\lambda_*\in\mathcal A$. This shows that $\mathcal A$ is open.

Let $\overline\lambda$ be an accumulation point of $\mathcal A$ that does not belong to $\mathcal A$,
and let $\overline{\mathbf g}\ne\mathbf g_{\overline\lambda}$ be a constant scalar curvature metric in the conformal class of $\mathbf g_{\overline\lambda}$
having volume equal to $v_{\overline\lambda}$. If $g_{\overline\lambda}$ were nondegenerate, then by the implicit function theorem (see Proposition~\ref{thm:localrigidity})
one could construct a differentiable path of constant scalar curvature metrics $\lambda\mapsto\mathbf h_\lambda$, $\lambda\in\left]\overline\lambda-\varepsilon,\overline\lambda+\varepsilon\right[$,
with $\mathbf h_{\overline\lambda}=\overline{\mathbf g}$, with $\mathbf h_\lambda\ne\mathbf g_\lambda$ in the conformal class of $\mathbf g_\lambda$
and of volume equal to $v_\lambda$ for all $\lambda$.
This contradicts the fact that for $\lambda\in\mathcal A$ near
$\overline\lambda$, $\mathbf g_\lambda$  is the unique such a metric in its conformal class.
\end{proof}
We have used a compactness result for solutions of the Yamabe problem:
\begin{prop}\label{thm:compactYamabe}
Let $M$ be a compact manifold and let $\mathcal K$ be a
set of smooth Riemannian metrics on $M$ which is compact in the $\mathcal C^{k,\alpha}$-topology with
$k$ \emph{sufficiently large}\footnote{%
Sufficiently large depending only on $\mathrm{dim}(M)$, see \cite[Lemma~10.1, p.\ 172]{KhuMarSch} for details. },
and such that one of the following assumptions is satisfied:
\begin{itemize}
\item[(a)] $\mathrm{dim}(M)\le7$;
\item[(b)] for all $\mathbf g\in\mathcal K$, then the Weyl tensor $W_\mathbf g$ of $\mathbf g$ satisfies \[\big\vert W_\mathbf g(p)\big\vert+\big\vert\nabla W_\mathbf g(p)\big\vert>0\]
at every point $p\in M$;
\item[(c)] $\mathrm{dim}(M)\le24$ and $M$ is spin.
\end{itemize}
Then, the set of unit volume constant scalar curvature metrics that belong to the conformal class of some $\mathbf g\in\mathcal K$ is compact
in the $\mathcal C^2$-topology. In particular, the conclusion holds for the family of metrics
$\mathcal K_n=\big\{\mathbf g_\lambda:\lambda\in\left[\tfrac{n-1}n,\frac n{n-1}\right]\big\}$
in the product $M=\mathds S^n\times\mathds S^n$.
\end{prop}
\begin{proof}
The result follows from the arguments in \cite{KhuMarSch, LiZhang, Mar}, see in particular \cite[Lemma10.1]{KhuMarSch}.
For the second statement, observe that the manifolds $(\mathds S^n\times\mathds S^n,\mathbf g_\lambda)$ satisfy assumption (b).
Namely, the Weyl tensor of $\mathbf g_\lambda$ is never vanishing in $\mathds S^n\times\mathds S^n$, since this is a homogeneous metric which is not locally conformally
flat for every $\lambda$. The given set $\mathcal K_n$ is compact in the $\mathcal C^{k,\alpha}$-topology for all $k$.
\end{proof}
In fact, the result of Proposition~\ref{thm:uniquenessprodspheres} extends immediately to the case
of products of arbitrary Einstein manifolds of positive scalar curvature. We need an elementary result first:
\begin{lem}\label{thm:weylproduct}
Let $W^{(0)}$, $W^{(1)}$ and $W$ be the Weyl tensors of $(M_0,\mathbf g^{(0)})$, $( M_1,\mathbf g^{(1)})$ and  $(M_0\times M_1,\mathbf g^{(0)}\oplus\mathbf g^{(1)})$ respectively. Assume that $M_0$ is Einstein at $p$ and $M_1$ is Einstein at $q$. Then $W$ vanishes at a point $(p,q)\in M_0\times M_1$ if and only if the following hold:
\begin{itemize}
\item[(a)] $W^{(0)}(p)=0$, $W^{(1)}(q)=0$,
\item[(b)] $m_1(m_1-1)\kappa^{(0)} + m_0(m_0-1)\kappa^{(1)}=0$, where $m_j=\dim(M_j)\ge2$ and $\kappa^{(j)}$ is the scalar curvature of $M_j$, $j=0,1$.
\end{itemize}
In particular, if both $\kappa^{(0)}$ and $\kappa^{(1)}$ are positive, then (b) is not satisfied and therefore $W(p,q)\ne0$.
\end{lem}
\begin{proof}
A direct elementary computation using the standard decomposition of a curvature tensor into its irreducible components, see for instance \cite{Besse}.
\end{proof}
A more general result that characterizes conformally flat product manifolds can be found in \cite[Theorem~4]{Yau73}.
\begin{prop}\label{thm:uniquenessproductEinstein}
Let $(M_0^{m_0},\mathbf g^{(0)})$ and $(M_1^{m_1},\mathbf g^{(1)})$ be compact Einstein manifolds of positive scalar curvature
$\kappa^{(0)}$ and $\kappa^{(1)}$ respectively.  Denote by $\mathbf g_\lambda$, $\lambda\in\left]0,+\infty\right[$, the metric
$\mathbf g^{(0)}\oplus\lambda\,\mathbf g^{(1)}$ on the product manifold $M=M_0\times M_1$.
Then, there exists an open subset $\mathcal A$ of $\left]0,+\infty\right[$ containing
$\lambda_*=\frac{m_0\kappa^{(1)}}{m_1\kappa^{(0)}}$ such that for all $\lambda\in\mathcal A$,
$\mathbf g_\lambda$ is the unique constant scalar curvature metric
in its conformal class, up to homotheties.

If $\overline\lambda$ is an accumulation point of $\mathcal A$,
then every constant curvature metric in the conformal class of $\mathbf g_{\overline\lambda}$ which is not homothetic
to $\mathbf g_{\overline\lambda}$ is degenerate.
\end{prop}
\begin{proof}
The proof of Proposition~\ref{thm:uniquenessprodspheres} can be repeated \emph{verbatim} here, observing that
the value $\lambda_*=\frac{m_0\,\kappa^{(1)}}{m_1\,\kappa^{(0)}}$ corresponds to the unique Einstein metric of the family
$\mathbf g_\lambda$. As to the compactness, note that assumption (b) of Proposition~\ref{thm:compactYamabe} is always satisfied
in products of Einstein manifolds with positive scalar curvature, by Lemma~\ref{thm:weylproduct}.
\end{proof}
\subsection{The case of non positive scalar curvature}
Let us now study the bifurcation problem for the family $\mathbf g_\lambda$ of metrics on the product
$M_0\times M_1$ under the assumption that either $\kappa^{(0)}$ or $\kappa^{(1)}$ are non positive.
First, we observe that if both $\kappa^{(0)}$ and $\kappa^{(1)}$ are non positive, then the pair $\big(\mathbf g^{(0)},\mathbf g^{(1)}\big)$ is nondegenerate.
If $\kappa^{(0)}\le0$ and $\kappa^{(1)}>0$, then the pair $\big(\mathbf g^{(0)},\mathbf g^{(1)}\big)$ is degenerate if and only if $\kappa^{(0)}=0$
and $\rho^{(1)}_{j_*}=\frac{\kappa^{(1)}}{m-1}$.
\begin{teo}\label{thm:nonpositivecase}
If $\kappa^{(0)}\le0$ and $\kappa^{(1)}\le0$, then the family $\mathbf g_\lambda$ has no degeneracy instants, and thus it is locally rigid
at every $\lambda\in\left]0,+\infty\right[$.

If $\kappa^{(0)}\le0$, $\kappa^{(1)}>0$ and the pair $\big(\mathbf g^{(0)},\mathbf g^{(1)}\big)$ is nondegenerate, then the
set of degeneracy instants for the Jacobi operator $\mathcal J_\lambda$ is a strictly decreasing sequence $\lambda_n$ that converges to $0$ as $n\to\infty$.
Moreover, every degeneracy instant is a bifurcation instant for the family $(\mathbf g_\lambda)_\lambda$.

Symmetrically,  if $\kappa^{(0)}>0$, $\kappa^{(1)}\le0$ and the pair $\big(\mathbf g^{(0)},\mathbf g^{(1)}\big)$ is nondegenerate, then the
set of degeneracy instants for the Jacobi operator $\mathcal J_\lambda$ is a strictly increasing unbounded sequence $\lambda_n$,
and every degeneracy instant is a bifurcation instant for the family $(\mathbf g_\lambda)_\lambda$.
\end{teo}
\begin{proof}
Follows from an elementary analysis of the zeroes of the functions $\sigma_{i,j}(\lambda)$ given in  \eqref{eigenij}. In the first case $\sigma_{i,j}(\lambda)>0$
for all $i$, $j=0,1,\dots$, $i+j\neq 0$. In the second (resp.\ in the third) one, the function $\sigma_{i,j}(\lambda)$ admits a zero for all $i\ge0$, and for
$j\in\{0,1,\ldots, j_*-1\}$ (resp., for all $j\ge0$, and for $i\in\{0,1,\ldots, i_*-1\}$) . Then we have a sequence of instants $(\lambda_n)_n$, that converges
to $0$ (resp. to $+\infty$) as $n\to\infty$ (see the proof of Corollary~\ref{thm:cordeginstants}), at each of which there is a jump in the dimension of the
negative eigenspace of $\mathcal J_\lambda$. The conclusion follows from Theorem~\ref{thm:mainbif}.
\end{proof}

\subsection{A multiplicity result in conformal classes of the bifurcating branches}
\label{sub:multiplicity}
Let us consider the case of constant scalar curvature manifolds $(M_0,\mathbf g^{(0)})$ and $(M_1,\mathbf g^{(1)})$, with $\kappa^{(1)}>0$,
and consider the product manifold $M=M_0\times M_1$ endowed with the family of metrics $\mathbf g_\lambda=\mathbf g^{(0)}\oplus\lambda\,\mathbf g^{(1)}$.
Let us recall the following terminology. A unit volume metric $\mathbf g$ on $M$ is a \emph{Yamabe} metric if it has constant scalar curvature, and
it realizes the minimum of all the scalar curvature among the unit volume constant scalar curvature in its conformal class.
Let $\mathcal Y(M)$ denote the \emph{Yamabe invariant} of $M$; recall that this is the supremum of the scalar curvature of all Yamabe metrics
of $M$. It is well known that $\mathcal Y(M)\le\mathcal Y(\mathds S^m)$.
\begin{prop}\label{thm:3CSC}
Let $\lambda_n$ be the decreasing sequence of bifurcation instants for the family $\mathbf g_\lambda$, with $\lim\limits_{n\to\infty}\lambda_n=0$.
Then, for $n$ sufficiently large, the conformal class of each metric in the branch bifurcating from
$\mathbf g_{\lambda_n}$ contains at least
\emph{three} distinct unit volume constant scalar curvature metrics.
\end{prop}
\begin{proof}
Since $\kappa^{(1)}>0$, one has $\lim\limits_{\lambda\to0^+}\kappa_\lambda=+\infty$, see \eqref{eq:kappalambda}.
Thus, for $\lambda>0$ sufficiently small, $\kappa_\lambda>\mathcal Y(\mathds S^m)\ge\mathcal Y(M)$, which implies that for
$\lambda$ small enough, $\mathbf g_\lambda$ is not a Yamabe metric. Thus, for $n$ large, $\mathbf g_{\lambda_n}$ is not a Yamabe metric,
and by continuity also nearby metrics are not Yamabe. Hence, each conformal class of the bifurcating branch issuing from
  $\mathbf g_{\lambda_n}$ contains a constant scalar curvature of the family, another distinct constant scalar curvature near by,
  and a Yamabe metric.
\end{proof}
\end{section}

\appendix
\begin{section}{Fiberwise implicit function theorem and bifurcation}\label{sec:appfiber}
In this appendix we give a formal statement of an implicit function theorem and two bifurcation results
for functions defined on the total space of a fiber bundle. Their proof is obtained readily from standard
results, and they will be omitted.
\subsection{Implicit function theorem}\label{sub:fiberIFT}
Given fiber bundles $\pi_i:E_i\to B_i$, $i=1,2$, and a $\mathcal C^1$-morphism of fiber bundles $\mathrm M:E_1\to E_2$,
the \emph{vertical derivative of $\mathrm M$ at $e\in E_1$} is the linear map \[\mathrm d_{\mathrm{ver}}\mathrm M(e):T_e\mathcal F(e)\to T_{\mathrm M(e)}\mathcal F\big(\mathrm M(e)\big)\]
given by the differential of the restriction $\mathrm M\big\vert_{\mathcal F(e)}:\mathcal F(e)\to\mathcal F\big(\mathrm M(e)\big)$, where $\mathcal F(e)=\pi_1^{-1}\big(\pi_1(e)\big)\subset E_1$
is the fiber of $E_1$ through the point $e$, and
$\mathcal F\big(\mathrm M(e)\big)=\pi_2^{-1}\big(\pi_2(M(e))\big)\subset E_2$ is the fiber of $E_2$ through $\mathrm M(e)$.\smallskip

We have used in the proof of Proposition~\ref{thm:localrigidity} a sort of \emph{fiber bundle implicit function theorem}, whose statement is as follows:
\begin{prop}\label{thm:fiberIFT}
Let $\pi_i:E_i\to B$, $i=1,2$, be fiber bundles, let $\mathrm M:E_1\to E_2$ be a fiber bundle morphism of class $\mathcal C^{k}$, $k\ge1$, let $s:U\subset B\to E_2$ be a local section of $E_2$
of class $\mathcal C^{k}$, with $U$ open subset
of $B$ containing $x_0$, $s(x_0)=e_2$, and let $e_1\in\mathrm M^{-1}(e_2)$. Assume that the vertical derivative
$\mathrm d_{\mathrm{ver}}\mathrm M(e_1)$ is an isomorphism. Then, there exists an open neighborhood $V$ of $e_1$ in $E_1$, with $U'=\pi_1(V)\subset U$, and a $\mathcal C^k$-section
$\tilde s:U'\to E_1$ with $\tilde s(x_0)=e_1$, such that $e\in V\cap \mathrm M^{-1}\big(s(U)\big)$ if and only if $e\in\tilde s(U')$.\qed
\end{prop}
\subsection{Fiberwise bifurcation}
\label{sub:fiberbif}
We propose a slightly more general statement of a celebrated bifurcation result by Smoller and Wasserman, see \cite{SmoWas}.
Recall that the basic setup of \cite{SmoWas} consists of a path $\lambda\mapsto\mathrm M_\lambda$ of \emph{gradient operators}
from a \emph{fixed} Banach space $B_2$ to another \emph{fixed} Banach space $B_0$, with $B_2\subset B_0$, and a path $\lambda\to u_\lambda\in B_2$
satisfying $M_\lambda(u_\lambda)=0$ for all $\lambda$. The main results in \cite{SmoWas} give sufficient conditions for the
existence of bifurcation branch of solutions of the equation $F(u,\lambda)=\mathrm M_\lambda(u)=0$ issuing from some point
of the path $u_\lambda$, both in the general and in the equivariant case. These results are used in the present paper in
a slightly different context, in that our setup consists of a gradient operators $F_\lambda$ defined on a smoothly
\emph{varying} Banach submanifold $\mathcal D_\lambda$ of a fixed Banach space, and taking values also in a smoothly \emph{varying}
family $\mathcal E_\lambda$ of closed subspaces of a Banach space. An extension of the results in \cite{SmoWas} to this situation
is quite straightforward, using local charts and projections, nevertheless it may be interesting to provide a
precise statement of the result which is employed in the present paper.

Let us give a few definitions. Given a Banach space $B$, a family $[a,b]\ni\lambda\mapsto B_\lambda$ of Banach submanifolds of $B$
is said to be a $\mathcal C^1$-family of submanifolds of $B$ if the set $\mathcal B=\big\{(x,\lambda)\in B\times[a,b]:x\in B_\lambda\big\}$
has the structure of a $\mathcal C^1$-sub-bundle of the trivial bundle $B\times[a,b]$ over $[a,b]$. For instance,
given a $\mathcal C^1$-function $f:B\times[a,b]\to\mathds R$ such that $\frac{\partial f}{\partial x}\ne0$ at all points in
$f^{-1}(0)$, then the family $B_\lambda=\big\{(x,\lambda):f(x,\lambda)=0\big\}$ is a $\mathcal C^1$-family of submanifolds of $B$.
Similarly, by a $\mathcal C^1$-family of closed subspaces of the Banach space $B$ we mean a family
$[a,b]\ni\lambda\mapsto S_\lambda$ of Banach subspaces of $B$ such that the set $\mathcal S=\big\{(x,\lambda):\lambda\in[a,b],\ x\in S_\lambda\big\}$
is a sub-bundle of the trivial Banach space bundle $B\times[a,b]$ over $[a,b]$. If $\lambda\mapsto x_\lambda\in B$ is a $\mathcal C^1$-path,
$\mathcal B=\bigcup_\lambda\big(B_\lambda\times\{\lambda\}\big)$ is a $\mathcal C^1$-family of submanifolds of $B$, with $x_\lambda\in B_\lambda$ for all $\lambda$,
then the path $\lambda\mapsto T_{x_\lambda}S_\lambda$ is a $\mathcal C^1$-family of closed subspaces of $B$.
\begin{teo}\label{thm:extSmoWasgeneral}
Let $B_0,B_2$ be Banach spaces, $H$ a Hilbertable space. Let $[a,b]\ni\lambda\mapsto\mathcal D_\lambda\subset B_2$
be a $\mathcal C^1$-family of submanifolds of $B_2$, and let $[a,b]\ni\lambda\mapsto\mathcal E_\lambda\subset B_0$ and
$[a,b]\ni\lambda\mapsto H_\lambda\subset H$ be $\mathcal C^1$-families of closed
subspaces of $B_0$ and of $H$ respectively.
 Let $F:\mathcal D\to\mathcal E$ be a $\mathcal C^1$ bundle morphism, and assume that the following are satisfied:
\begin{itemize}
\item[(a)] $\lambda\mapsto e_\lambda\in\mathcal E_\lambda$ is a $\mathcal C^1$-section of the bundle $\mathcal E$;
\item[(b)] $\lambda\mapsto d_\lambda\in\mathcal D_\lambda$ is a $\mathcal C^1$-section of the bundle $\mathcal D$, with \[F(d_\lambda,\lambda)=(e_\lambda,\lambda)\] for all $\lambda$;
\item[(c)] it is given a $\mathcal C^1$-family of complete inner products $\lambda\mapsto\langle\cdot,\cdot\rangle_\lambda$ in $H_\lambda$;
\item[(d)] there are continuous inclusions $B_2\subset B_0\subset H$ that induce inclusions $T_{d_\lambda}\mathcal D_\lambda\subset\mathcal E_\lambda\subset H_\lambda$
for all $\lambda$;
\item[(e)] for all $\lambda$, the map $F_\lambda=F(\cdot,\lambda):\mathcal D_\lambda\to\mathcal E_\lambda$ is a gradient operator at $d_\lambda$, i.e.,
the differential $\mathrm dF(\cdot,\lambda):T_{d_\lambda}\mathcal D_\lambda\to\mathcal E_\lambda$ is symmetric relatively to
the inner product $\langle\cdot,\cdot\rangle_\lambda$;
\item[(f)] $\mathrm dF(\cdot,\lambda):T_{d_\lambda}\mathcal D_\lambda\to\mathcal E_\lambda$ is Fredholm of index $0$ for all $\lambda$;
\item[(g)] for all $\lambda$, there exists an $\langle\cdot,\cdot\rangle_\lambda$-orthonormal basis $e_1^\lambda,e_2^\lambda,\ldots$ of $H_\lambda$ consisting of eigenvectors
of $\mathrm dF(\cdot,\lambda)$;
\item[(h)] the corresponding eigenvectors have finite multiplicities, and for all $\lambda$ the number $n_\lambda$ of eigenvalues (counted with multiplicities)
of $\mathrm dF(\cdot,\lambda)$ that are negative is finite;
\item[(i)] there exists $\lambda_*\in\left]a,b\right[$ such that, for $\varepsilon>0$ sufficiently small:
 \begin{itemize}
 \item $\mathrm dF(\cdot,\lambda_*-\varepsilon)$ and $\mathrm dF(\cdot,\lambda_*+\varepsilon)$ are non singular;
\item $n_{\lambda_*-\varepsilon}\ne n_{\lambda_*+\varepsilon}$.
\end{itemize}
\end{itemize}
Then,  $\lambda_*$  is a \emph{bifurcation instant} for the equation \[F(\cdot,\lambda)=(e_\lambda,\lambda),\] i.e.,
there exists a sequence $d_n\in B_2$, and a sequence $\lambda_n$ in $[a,b]$, with $d_n\in\mathcal D_{\lambda_n}$ for all $n$, $\lim\limits_{n\to\infty}\lambda_n=\lambda_*$, $\lim\limits_{n\to\infty}d_n=d_{\lambda_*}$, $d_n\ne d_{\lambda_n}$ for all $n$, and such that \[F(d_n,\lambda_n)=(e_{\lambda_n},\lambda_n)\] for all $n$.
\end{teo}
\begin{proof}
Sufficiently small neighborhoods of $(d_{\lambda_*},\lambda_*)$ in $\mathcal D$ and of $(e_{\lambda_*},\lambda_*)$ in $\mathcal E$ are identified
respectively with open subsets of  products $T_{d_{\lambda_*}}\mathcal D_{\lambda_*}\times[\lambda_*-\varepsilon,\lambda_*+\varepsilon]$ and $\mathcal E_{\lambda_*}\times[\lambda_*-\varepsilon,\lambda_*+\varepsilon]$.
Using these identifications, the bundle morphism $F$ is given by a $\mathcal C^1$ path of gradient operators $F_\lambda$ between open subsets of the
Banach spaces $T_{d_{\lambda_*}}D_{\lambda_*}$ and $\mathcal E_{\lambda_*}$. The result is then obtained as a straightforward application
of \cite[Theorem~2.1]{SmoWas}.
\end{proof}
In the situation described by items (a)---(h) in Theorem~\ref{thm:extSmoWasgeneral},
assume that $G$ is a connected (or more generally, a \emph{nice} in the sense of \cite{SmoWas}) Lie group,
and that $B_0$, $B_2$ and $H$ are $G$-spaces. Assume that $\mathcal D_\lambda$, $\mathcal E_\lambda$ and $H_\lambda$ are $G$-invariant for
all $\lambda$, and that $F$ is $G$-equivariant, i.e.:
\[F(g\cdot d,\lambda)=g\cdot F(d,\lambda)\]
for all $(d,\lambda)\in\mathcal D$ and all $g\in G$.
Assume further that $g\cdot d_\lambda=d_\lambda$ and $g\cdot e_\lambda=e_\lambda$ for all $g\in G$ and all $\lambda$.
It is easy to see that every eigenspace of $\mathrm dF(\cdot,\lambda)$ is $G$-invariant for all $\lambda$. Denote by $\pi_\lambda^-$ the
representation of $G$ on the finite dimensional space given by the direct sum of all eigenspaces of  $\mathrm dF(\cdot,\lambda)$ corresponding
to negative eigenvalues.
\begin{teo}\label{thm:extSmoWasequivariant}
Let $\lambda_*\in\left]a,b\right[$ be such that, for $\varepsilon>0$ sufficiently small:
\begin{itemize}
\item $\mathrm dF(\cdot,\lambda_*-\varepsilon)$ and $\mathrm dF(\cdot,\lambda_*+\varepsilon)$ are non singular;
\item $\pi_{\lambda_*-\varepsilon}^-$ and $\pi_{\lambda_*+\varepsilon}^-$ are not equivalent.
\end{itemize}
Then,  $\lambda_*$ is a bifurcation instant for the equation $F(\cdot,\lambda)=(e_\lambda,\lambda)$.
\end{teo}
\begin{proof}
The result is an application of \cite[Theorem~3.1]{SmoWas}, using a local product structure of $\mathcal D$ and $\mathcal E$ around the points
$(d_{\lambda_*},\lambda_*)$ and $(e_{\lambda_*},\lambda_*)$.
\end{proof}
\end{section}

\end{document}